\documentclass[letterpaper]{amsart}

\usepackage{amsmath,amssymb}
\usepackage{color}

\newtheorem{theorem}{Theorem}
\newtheorem{corollary}[theorem]{Corollary}
\newtheorem{condition}[theorem]{Condition}
\newtheorem{prop}[theorem]{Proposition}
\newtheorem{lem}[theorem]{Lemma}
\theoremstyle{definition}
\newtheorem{definition}{Definition}
\newtheorem{ex}{Example}
\theoremstyle{remark}
\newtheorem*{remark*}{Remark}

\DeclareMathOperator{\Perm}{Perm}
\DeclareMathOperator{\Stab}{Stab}
\newcommand{\nat}{\mathbb N}
\newcommand{\rat}{\mathbb Q}
\newcommand{\simc}{\sim_{\mathbf c}}
\newcommand{\simb}{\sim_{\mathbf b}}
\newcommand{\simh}{\sim_{\mathbf h}}
\newcommand{\Inc}{\mbox{\rm Inc}}

\usepackage{yfonts}
\usepackage{array}

\newcommand{\MMF}{M. Mend\`es France}
\newcommand{\SB}{S. Brlek}
\newcommand{\JPL}{J.-P. Labb\'e}

\begin{document}


\title{Combinatorial variations on Cantor's diagonal}
 \author[S. Brlek]{Sre\v{c}ko Brlek}
\address{\SB, LaCIM, Universit\'e du Qu\'ebec \`a Montr\'eal, C.P. 8888, Succ. Centre-ville,
   Montr{\'e}al (QC) Canada H3C 3P8}
 \email{brlek.srecko@uqam.ca}
 \author[J.-P. Labb\'e]{Jean-Philippe Labb\'e}
\address{\JPL, Freie Universit\"at Berlin,
    Arnimallee 2, 14195 Berlin, Deutschland}
    \email{labbe@math.fu-berlin.de}
     \author[M.Mend\`es France]{Michel Mend\`es France}
 \address{\MMF, D\'epartement de 
   math\'ematiques, UMR 5251, Universit\'e Bordeaux 1, 351 cours de la Lib\'eration, F-33405 Talence  cedex, France }
  \email{mmf@math.u-bordeaux1.fr}


\begin{abstract}

We discuss counting problems linked to finite versions of Cantor's diagonal of infinite tableaux. We extend previous results  of \cite{BMRR} by refining an equivalence relation that reduces significantly the exhaustive generation. New enumerative results follow and allow to look at the sub-class of the so-called {\em bi-Cantorian tableaux}. We conclude with a correspondence between Cantorian-type tableaux and coloring of hypergraphs having a square number of vertices.


\end{abstract}

\maketitle


\section{Introduction}\label{sec:Intro}

In a celebrated paper, Cantor  \cite{cantor} proved the existence of transcendental numbers using his famous diagonal argument, based on the comparison of the set of rows of an infinite tableau with its diagonal.
It amounts to fill a countable infinite tableau with a list of algebraic numbers in  base 2, and to compare with a word built such that  for each index $i$, its $i$th  digit is different from the diagonal's $i$th digit. Indeed such a word does not appear on any row of the infinite tableau. 
Since no assumption is made on the ordering, any permutation of the rows yields the same conclusion. In fact, it can be shown that the diagonal itself does not appear in any row of the tableau, provided that each rational number appears twice, once with trailing $0$'s and once with trailing $1$'s as shown  in Brlek, Mend\`es France, Robson and Rubey \cite{BMRR}. They proceed as follows. On a finite alphabet $A=\{\alpha_1, \alpha_2,\dots,\alpha_s\}$, the \emph{permanent}
of an infinite tableau $T:\nat\times\nat\to A$ is the set of infinite sequences
\begin{equation*}
  \Perm(T)=\bigcup_{\pi\in S_{\nat}}
  a_{\pi(1)}^1a_{\pi(2)}^2a_{\pi(3)}^3\cdots,
\end{equation*}
where $S_{\nat}$ is the family of all bijections $\pi: \nat\rightarrow \nat$.
Therefore,  if the set $L$ of row-words of $T$ consists of all  algebraic numbers in the unit interval represented in base~$2$, then $\Perm (T)$ is exactly the set of all  transcendental numbers in the unit interval (\cite{BMRR} Corollary 12), and $\Perm (T) \cap L=\emptyset$. Such a tableau is called \emph{Cantorian} and has a natural finite counterpart:
the \emph{permanent} of a square~$n\times n$~tableau $T=(a_{i}^j)$ with $a_i^j\in A$ is the  set of words
\begin{equation}\label{Perm}
\Perm(T)  =  \bigcup_{\pi \in S_n} \left\{a_{\pi(1)}^1a_{\pi(2)}^2\cdots a_{\pi(n)}^n\right\},
\end{equation}
where~$S_n$ is the set of all permutations on $n$ elements. Thus, the permanent of a tableau~$T$ is nothing but the set of distinct diagonal-words obtained by permuting the rows of~$T$. Then,  a tableau~$T$ is \emph{Cantorian} if~$\Perm(T)\cap L=\emptyset$, where~$L$ is the set of distinct row-words of~$T$. This means that no row-words of~$T$ can appear as a diagonal of~$T$ by permuting its rows.  

In the first part of their paper, they gave a sufficient condition for a tableau to be Cantorian, and they also designed a polynomial time algorithm to test whether a tableau is Cantorian or not. Some  enumerative results for tableaux of small size were also provided, especially for tableaux of size smaller than $11$ on two letter alphabets, which reveals the computation bottleneck of the problem. In an attempt to ease the computation, they also introduced a natural equivalence relation on Cantorian tableaux at the end of their paper. 

In this paper, we study this equivalence relation in more detail and its consequences on the classification of Cantorian tableaux, their enumeration and the study of bi-Cantorian tableaux.  In Section  \ref{sec:Prelim}, we introduce  the combinatorial objects necessary for our purpose with  examples. Section  \ref{sec:EC} contains a class invariant for the equivalence relation and, using group actions, we get a formula for the cardinality of a class. Section  \ref{sec:Generation} describes the canonical representatives that are useful for computations, with new enumerative results obtained exhaustively, and closed formulas for Cantorian tableaux as well. Section  \ref{sec:BiC} is devoted to the investigation on bi-Cantorian tableaux. In Section  \ref{sec:Variations}, we study columns and extensions of infinite Cantorian-type tableaux related to number theory. Finally, in Section \ref{sec:HyperGraphs}, we describe a correspondence between Cantorian-type tableaux and colored hypergraphs.

\section{Preliminaries}\label{sec:Prelim}
We consider a finite and ordered alphabet  $A=\{\alpha_1, \alpha_2,\dots,\alpha_s\}$ where $s\geq 2$ and $\alpha_i<\alpha_j$ whenever $i<j$. Then ~$A^\star$ is the set of finite words over~$A$. The lexicographic order on~$A^\star$ is denoted  $\leq$.  The number of occurrences of the letter~$\alpha\in A$ in $w\in A^\star$ is denoted $|w|_\alpha$. We write~$\mathcal{T}^s_n$ for the set of all square~$n\times n$ tableaux $T$ with entries in~$A$. Recall that given a tableau $T$, $L$ is the set of row-words of $T$. The sequence of column-words of~$T$ is denoted by~$\overline{C}=(c_1,c_2,\dots,c_n)$ while the set of distinct column-words is denoted by~$C$. Then each of the sets~$L$ and~$C$ clearly contain at most~$n$ words.

It is convenient to build  the set~$\mathbb{N}^\star$ of sequences (or words!) of  natural numbers~$\mathbb{N}=\{0,1,2,3,\dots\}$ considered as an infinite alphabet. Elements~$\lambda$ of~$\mathbb{N}^\star$ are called  \emph{compositions} (the definition of composition is loosened by allowing null parts). It can be viewed as a function 
$\lambda : [0..(m-1)] \rightarrow \nat$, with $m\in \nat\setminus \{0\}$. The \emph{weight}~$|\lambda|\in\nat$ of a composition~$\lambda$ is the sum of the numbers appearing in it. The \emph{length}  is ~$\ell(\lambda)=m$. If~$\lambda$ is a composition of weight~$n$, we say that~$\lambda$ is a \emph{partition} of weight~$n$ if it is decreasing (not necessarily strictly). 
We denote by~$\mathcal{C}_n$ (respectively~$\mathcal{P}_n$) the set of compositions (respectively partitions) of weight~$n$. There is a natural projection~$\pi_n$ from~$\mathcal{C}_n$ to~$\mathcal{P}_n$ defined by ordering the compositions in decreasing order.

By abuse of notation, we denote the lexicographic order on~$\mathbb{N}^\star$ deduced from the natural order on~$\nat$ also by~$\leq$. We define a particular total order~$\preceq$ on~$\mathbb{N}^\star$,  referred to as the \emph{composition order}.

\begin{definition}
Let~$\lambda,\lambda' \in\mathbb{N}^\star$. We write~$\lambda \preceq \lambda'$ if and only if
\begin{displaymath}
 \ell(\lambda)<\ell(\lambda') \text{ or } \left( \ell(\lambda)=\ell(\lambda') \text{ and } \lambda'\leq\lambda\right).
\end{displaymath}
\end{definition}
Note the inversion in the condition $(\lambda'\leq\lambda)$, which says that the inverse lexicographic order is used. 
To illustrate the composition order, consider the set of \emph{positive} compositions (without $0$) of length at most~$3$ and weight~$5$. It is totally ordered by~$\preceq$ as follows
\[
 5 \preceq 41 \preceq 32 \preceq 23 \preceq 14 \preceq 311 \preceq 221 \preceq 212 \preceq 131 \preceq 122 \preceq 113.
\]
For computational purposes, the Cantorian representatives, defined later on, are built from partitions without zeros, therefore having different lengths; this justifies the condition $\ell(\lambda)<\ell(\lambda')$ in the definition of $\preceq$.

It is useful to encode words $A^\star$ by using compositions\footnote{This is an adaptation of the \emph{Parikh map}  (see Supplementary Lecture H of \cite{kozen}) sending a word to a composition of length~$s$ in place of a vector.} of length $s$.
\begin{definition}
Let~$w\in A^\star$ such that $|w|=n$. The \emph{Parikh composition}~$\mathfrak{p}_w:=\mathfrak{p}(w)$ of~$w$ is a composition of weight~$n$ and length~$s$ defined by the map $\mathfrak{p}: A^\star  \rightarrow  \mathbb{N}^\star $ 
\begin{equation*}
w  \mapsto  |w|_{\alpha_1}|w|_{\alpha_2}\cdots |w|_{\alpha_s}.
\end{equation*}
\end{definition}

\begin{ex}
 Let~$A=\{1,2,3\}$ and~$w=12323$,~$v=2233$ be words in~$A^\star$, then their Parikh compositions are~$\mathfrak{p}_{w}=122$ and~$\mathfrak{p}_{v}=022$.
\end{ex}

We use the composition order to define a total order~$\blacktriangleleft$ on~$A^\star$, called the \emph{Parikh composition order} on~$A^\star$.

\begin{definition}
 Let~$w,w' \in A^\star$. We write~$w \blacktriangleleft w'~$ if and only if
\begin{align*}
 \mathfrak{p}_{w}\prec\mathfrak{p}_{w'} \text{ or } \left( \mathfrak{p}_{w}=\mathfrak{p}_{w'} \text{ and } w\leq w'\right).
\end{align*}
\end{definition}

\begin{ex}
 Let~$A=\{1,2,3\}$. If~$w=12323,v=21233,u=32121$ are words in~$A^\star$, then~$\mathfrak{p}_{w}=122,\mathfrak{p}_{v}=122,\mathfrak{p}_{u}=221$ and~$u\blacktriangleleft w \blacktriangleleft v$.
\end{ex}

The map~$\mathfrak{p}$ is naturally extended to  a map~$\mathfrak{P}: \mathcal{T}^s_n \rightarrow (\mathbb{N}^\star)^n$. Then, the order~$\blacktriangleleft$ is extended to the set~$\mathcal{T}^s_n$ by looking at the image ~$\mathfrak{P}(T)$ from left to right.

\begin{definition}
The \emph{Parikh compositions} $\mathfrak{P}_{T}:=\mathfrak{P}(T)=(\mathfrak{p}_{c_1},\mathfrak{p}_{c_2},\dots,\mathfrak{p}_{c_n})$ of a tableau~$T$ is the vector of Parikh compositions of its column-words in~$\overline{C}$.
\end{definition}

\begin{ex} \label{ex:parikh_comp}
 Here are a few tableaux on the alphabet~$\{1,2,3\}$ with their Parikh compositions.
\[{\footnotesize
\begin{array}{rl}
T_1 =&\begin{bmatrix}
                          1 & 1 & 3 \\
                          1 & 1 & 2 \\
                          2 & 3 & 1
              \end{bmatrix},\medskip\\ 
\mathfrak{P}_{T_1} =&(210,201,111)
\end{array}
\begin{array}{rl}
T_2 = & \begin{bmatrix}
                          2 & 1 & 1 & 2 \\
                          3 & 1 & 2 & 1 \\
                          2 & 1 & 1 & 1 \\
                          2 & 1 & 2 & 1 \\
                         \end{bmatrix}, \medskip\\
\mathfrak{P}_{T_2} =&(031,400,220,310)
\end{array}
\begin{array}{rl}
T_3 =&  \begin{bmatrix}
                          1 & 2 & 3 & 1 & 2 & 2 \\
                          2 & 1 & 1 & 2 & 3 & 3 \\
                          3 & 1 & 2 & 1 & 1 & 2 \\
                          2 & 3 & 2 & 1 & 1 & 1 \\
                          3 & 3 & 2 & 1 & 2 & 1 \\
                          1 & 1 & 1 & 1 & 3 & 2 
                \end{bmatrix}.\medskip\\ 
\mathfrak{P}_{T_3} =&(222,312,231,510,222,231)
\end{array} 
}\]

\end{ex}

\section{Structure of equivalence classes} \label{sec:EC}

In  \cite{BMRR}, the authors introduced the following  equivalence relation on~$\mathcal{T}^s_n$.

\begin{definition}[Brlek et al.\cite{BMRR}]
Let~$T',T\in\mathcal{T}^s_n$. One writes ~$T'\simc T$  if and only if~$T'$ can be obtained from~$T$ by a finite combination of the following operations: permutation of rows, permutation of columns, replacing all entries of a column by their image under any bijection of the alphabet.
\end{definition}

\begin{ex}
The following tableaux are all equivalent and ordered in decreasing order according to~$\blacktriangleleft$. First, by applying the bijection~$1\rightarrow 2, 2\rightarrow 1, 3\rightarrow 3$ to the first column and the bijection~$1\rightarrow 2, 2\rightarrow 3, 3\rightarrow 1$ to the second column. Secondly, by swapping the second and third column. Next, by swapping the second and third row. Finally, by applying the bijection~$1\rightarrow 1, 2\rightarrow 3, 3\rightarrow 2$ to the third column.
\[{\footnotesize
          \begin{bmatrix}
                          2 & 3 & 1\\
                          2 & 2 & 2\\
                          2 & 3 & 1\\
                         \end{bmatrix} \blacktriangleright
         \begin{bmatrix}
                          1 & 1 & 1\\
                          1 & 3 & 2\\
                          1 & 1 & 1\\
                         \end{bmatrix} \blacktriangleright
          \begin{bmatrix}
                          1 & 1 & 1\\
                          1 & 2 & 3\\
                          1 & 1 & 1\\
                         \end{bmatrix} \blacktriangleright
        \begin{bmatrix}
                          1 & 1 & 1\\
                          1 & 1 & 1\\
                          1 & 2 & 3\\
                         \end{bmatrix} \blacktriangleright 
        \begin{bmatrix}
                          1 & 1 & 1\\
                          1 & 1 & 1\\
                          1 & 2 & 2\\
                         \end{bmatrix}
                         }
\]
\end{ex}

Given a tableau~$T$, we denote its equivalence class by~$[T]$ and its class cardinality by~$\#[T]$; it is clear that~$\#[T]\leq n!^2(s!)^n$. By inspection  it is often much less, but it is still a sharp bound, as shown in the next example.

\begin{ex}
Consider the three tableaux of Example \ref{ex:parikh_comp}. We have
\begin{align*}
 \#[T_1]=1944~ & \leq  ~7776=(3!)^2(3!)^3; \\
 \#[T_2]=24186470400~ & = ~(6!)^2(3!)^6; \\
 \#[T_3]=186624~ & \leq  ~746496=4!^2(3!)^4. 
\end{align*}
\end{ex}
Given any two tableaux~$T$ and~$T'$, deciding whether~$T$ is equivalent or not to~$T'$ can be a tedious task if we inspect all combinations  of permutations and bijections. To accelerate this process, we introduce  a class invariant of tableaux. Recall that~$\pi_n$ is the projection from the compositions of weight~$n$ to the partitions of weight~$n$. Let~$\pi^n_n$ be the Cartesian product~$\underbrace{\pi_n\times\pi_n\times\cdots\times\pi_n}_{n\text{ times}}$ and~$\Inc_{\preceq}:(\mathcal{C}_n)^n \longrightarrow (\mathcal{C}_n)^n$ the function that reorders vectors in increasing order according to the order~$\preceq$.

\begin{lem}
The function~$\Inc_\preceq \circ \pi_n^n \circ \mathfrak{P}$ is a class function with respect to the relation~$\simc$ over~$\mathcal{T}^s_n$.
\end{lem}

\begin{proof}
 Take a tableau~$T$ and consider its equivalence class~$[T]$. First note that any tableau~$T'\in[T]$ obtained by permuting rows shares the same Parikh compositions as ~$T$, i.e.~$\mathfrak{P}(T)=\mathfrak{P}(T')$. This means that applying~$\Inc_\preceq \circ \pi_n^n \circ \mathfrak{P}$ to both~$T$ and~$T'$ yields the same result. Then, let~$T''\in[T]$ be a tableau obtained by applying some bijections to the columns of~$T$. By permuting the letters of the compositions of~$T$ along the associated bijections, we get the compositions of~$T''$. Then, applying~$\pi^n_n$ will reorder all compositions in~$\mathfrak{P}(T)$ and~$\mathfrak{P}(T'')$ decreasingly to produce the same Parikh partitions, hence~$\pi_n^n \circ \mathfrak{P}(T)=\pi_n^n \circ \mathfrak{P}(T'')$. Finally, let~$T'''$ be a tableau obtained by a permutation of columns of~$T$. The Parikh partitions~$\pi_n^n \circ \mathfrak{P}(T)$ and~$\pi_n^n \circ \mathfrak{P}(T''')$ are just a permutation apart . Applying~$\Inc_\preceq$ reorders the partitions increasingly according to~$\preceq$, thus~$\Inc_\preceq \circ \pi_n^n \circ \mathfrak{P}(T)=\Inc_\preceq \circ \pi_n^n \circ \mathfrak{P}(T''')$. Finally, since any tableau equivalent to~$T$ is obtained by a finite sequence of permutations of rows, columns and bijection of columns, the result follows.
\end{proof}

\begin{definition}
 The vector ~$\Inc_{\preceq}\circ\pi^n_n(\mathfrak{P}_{T})$ is called the \emph{Parikh compositions representative} of ~$T$.
\end{definition}

Now, given two tableaux and their Parikh compositions, we have a necessary condition for them to be equivalent.

Moreover, the cardinality of each class~$[T]$ can be established  without actually computing all its equivalent tableaux by only considering its internal structure. This is of particular interest for generating the Cantorian tableaux. In order to establish  the formula, we need some technical results.
Recall that  $S_n$ is the group of permutations on $n$ elements 
with identical element  $e$.

\begin{lem}
The group~$\mathbb{S}:=S_n\times S_n$ acts on~$\mathcal{T}^s_n$ by permutation of rows and  columns:
\begin{align*}
\Phi: \mathbb{S} \times \mathcal{T}^s_n & \longrightarrow  \mathcal{T}^s_n \\ 
(\sigma,\tau,T) & \longmapsto  \sigma T \tau^{-1}.
\end{align*}
\end{lem}

\begin{proof}
The two group axioms are easily verified: $\Phi(e,e,T)=eTe=T$, for all~$T\in\mathcal{T}^s_n$ and~$\Phi(\sigma_1\sigma_2,\tau_1\tau_2,T)=\sigma_1\sigma_2T(\tau_1\tau_2)^{-1}=\sigma_1\sigma_2T\tau_2^{-1}\tau_1^{-1}=\Phi(\sigma_1,\tau_1,\Phi(\sigma_2,\tau_2,T))$ for all~$(\sigma_1,\tau_1),(\sigma_2,\tau_2)\in\mathbb{S}$ and~$T\in\mathcal{T}^s_n$.
\end{proof}

\begin{lem}\label{lem:orbit_perm}
 The cardinality  of the orbit~$\mathcal{O}_\Phi(T)$ of~$T$ through the action~$\Phi$ is
\begin{equation}\label{eq:orbit_perm}
 |\mathcal{O}_\Phi(T)|=\frac{|\mathbb{S}|}{|\Stab_\Phi(T)|}=\frac{(n!)^2}{\left(\prod_{j=1}^rg_j!\prod_{i=1}^q{f_i!}+ \eta\right)},
\end{equation}
 where~$(f_1,f_2,\dots,f_q)$ is the vector of multiplicities of row-words of~$T$, $(g_1,g_2,\dots,g_r)$ its column-words multiplicities, $\eta=|\{(\sigma,\tau)\in \mathbb{S}\mid \sigma T\tau^{-1}= T \text{ and } \sigma T\neq T\}|$, $q=|L|$ and $r=|C|$.
\end{lem}

\begin{proof}
Using the  orbit-stabilizer theorem for group actions, we only need to find the cardinality of the stabilizer of a tableau~$T$. Indeed, if two or more rows are equal, then any permutation within these rows does not alter~$T$. 
The subgroup of row permutations that stabilizes~$T$ is a Young subgroup of cardinality~$\prod_{i=1}^q{f_i!}$. Similarly for columns, the subgroup of column permutations that stabilizes~$T$ is a Young subgroup of cardinality~$\prod_{j=1}^rg_j!$. Then, it might happen that a combination of column and row permutations still stabilizes~$T$, without actually fixing~$T$ when acting separately. The set of such pair is~$\{(\sigma,\tau)\in \mathbb{S}\mid \sigma T\tau^{-1}= T \text{ and } \sigma T\neq T\}$.
\end{proof}

\begin{lem}
The group~$\mathbb{B}:=(S_s)^n$ acts on the right on~$\mathcal{T}^s_n$ by bijection of columns:
\begin{align*}
\Psi: \mathcal{T}^s_n \times \mathbb{B}  & \longrightarrow  \mathcal{T}^s_n \\ (T,(\beta_1,\dots,\beta_n)) & \longmapsto  T\cdot(\beta_1,\dots,\beta_n).
\end{align*}
\end{lem}

\begin{proof}
This is straightforward.
\end{proof}

\begin{lem}\label{lem:orbit_bij}
Let~$T\in\mathcal{T}^s_n$ have the Parikh compositions~$\mathfrak{P}_T=(\mathfrak{p}_{c_1},\mathfrak{p}_{c_2},\dots, \mathfrak{p}_{c_n})$. The cardinality of the orbit~$\mathcal{O}_\Psi(T)$ of~$T$ under the action~$\Psi$ is
\begin{equation}\label{eq:orbit_bij}
|\mathcal{O}_\Psi(T)|=\frac{|\mathbb{B}|}{|\Stab_{\Psi}(T)|}=\prod_{i=1}^n\frac{s!}{(s-\ell^+(\mathfrak{p}_{c_i}))!},
\end{equation}
where~$\ell^+(\mathfrak{p}_{c_i})$ is the number of non-zero letters in~$\mathfrak{p}_{c_i}$.
\end{lem}

\begin{proof}
Again, we use  the orbit-stabilizer theorem. Consider the $i$th column-word~$c_i$ of the tableau~$T$. We must count the number of bijections on $A$ that fix~$c_i$. Then, each letter in~$c_i$ must remain fixed, while  all other letters might be permuted. This yields ~$(s-\ell^+(\mathfrak{p}_{c_i}))!$ distinct bijections, where~$\ell^+(\mathfrak{p}_{c_i})$ is the number of distinct letters in~$c_i$, i.e. the number of non-zero letters in~$\mathfrak{p}_{c_i}$. Then, taking the product for all~$1\leq i \leq n$, we get the desired formula.
\end{proof}

\begin{theorem}\label{thm:card_class}
Let~$T\in\mathcal{T}^s_n$ have the Parikh compositions~$\mathfrak{P}_T=(\mathfrak{p}_{c_1},\mathfrak{p}_{c_2},\dots, \mathfrak{p}_{c_n})$. Let~$(f_1,f_2,\dots,f_q)$,~$(g_1,g_2,\dots,g_r)$, $\eta$, $q$ and~$r$ be defined as in Lemma $\ref{lem:orbit_perm}$ and~$\ell^+(\mathfrak{p}_{c_i})$ defined as in Lemma $\ref{lem:orbit_bij}$. The cardinality of~$[T]$ is
\begin{equation}
\#[T]=\frac{|\mathcal{O}_\Phi(T)|\cdot|\mathcal{O}_\Psi(T)|}{\vartheta}  =  \frac{(n!)^2\prod_{i=1}^n\frac{s!}{(s-\ell^+(\mathfrak{p}_{c_i}))!}}{\left(\prod_{j=1}^rg_j!\prod_{i=1}^q{f_i!}+ \eta\right)\vartheta}, \label{eq:card_class}
\end{equation}
where~$\vartheta=|\mathcal{O}_\Psi(T)\cap\mathcal{O}_\Phi(T)|$.
\end{theorem}

\begin{proof}
The semi-direct product~$\mathbb{B}\rtimes\mathbb{S}$ acts on $\mathcal{T}^s_n$ as follows:
\begin{align}\label{eq:action}
\Omega: \mathcal{T}^s_n \times (\mathbb{B}\rtimes\mathbb{S})  & \longrightarrow  \mathcal{T}^s_n,  \\ (T,(\beta,\sigma)) & \longmapsto  T\cdot(\beta,\sigma)=T\beta\sigma \nonumber
\end{align}
The action  $\beta$ consists in applying a set of bijections $(\beta_1, \beta_2, \dots, \beta_n)$ on $A$ respectively on  $(c_1, c_2, \dots, c_n)$. Writing~$T\beta\sigma$ means that~$\beta$ acts first, and then~$\sigma$ permutes the rows and columns. It is obvious that~$T\cdot(e,e)=T$, for all~$T\in\mathcal{T}^s_n$. Then, given~$(\beta_1,\sigma_1),(\beta_2,\sigma_2)\in(\mathbb{B}\rtimes\mathbb{S})$, we have
\begin{align*}
T\cdot((\beta_1,\sigma_1)*(\beta_2,\sigma_2)) & =  T\cdot(\beta_1\sigma_1\beta_2\sigma_1^{-1},\sigma_1\sigma_2)\\
 & =  T\beta_1\sigma_1\beta_2\sigma_1^{-1}\sigma_1\sigma_2 \\
  & =  T\beta_1\sigma_1\beta_2\sigma_2 \\
  & =  (T\beta_1\sigma_1)\beta_2\sigma_2 \\
  & =  (T\cdot(\beta_1,\sigma_1))\beta_2\sigma_2 \\
  & =  (T\cdot(\beta_1,\sigma_1))\cdot(\beta_2,\sigma_2)\;. 
\end{align*}
Thus, this is a valid group action. Then, using the orbit-stabilizer theorem, we get
\begin{equation*}
|\mathcal{O}_\Omega(T)|  =  \frac{|\mathbb{B}\rtimes\mathbb{S}|}{|\Stab_{\Omega}(T)|} = \frac{|\mathbb{B}|\cdot|\mathbb{S}|}{|\Stab_{\Psi}(T)|\cdot|\Stab_{\Phi}(T)|\cdot|\Stab_{\Psi\Delta\Phi}(T)|},
\end{equation*}
where~$\Stab_{\Psi\Delta\Phi}(T)=\{(\beta,\sigma)\in\mathbb{B}\rtimes\mathbb{S}| T\beta\sigma=T\text{ and }T\beta \neq T\}\cup\{(e,e)\}$. A simple computation shows that this last set is in bijection with ~$\mathcal{O}_\Psi(T)\cap\mathcal{O}_\Phi(T)$, so that 
\begin{equation*}
\frac{|\mathbb{B}|\cdot|\mathbb{S}|}{|\Stab_{\Psi}(T)|\cdot|\Stab_{\Phi}(T)|\cdot|\Stab_{\Psi\Delta\Phi}(T)|}  =  \frac{|\mathcal{O}_\Phi(T)|\cdot|\mathcal{O}_\Psi(T)|}{|\mathcal{O}_\Psi(T)\cap\mathcal{O}_\Phi(T)|},
\end{equation*}
and using Equations \eqref{eq:orbit_perm} and \eqref{eq:orbit_bij} yields what we claimed.
\end{proof}

\begin{remark*}
 On one hand, this equation leads to a closed formula for Cantorian tableaux of small dimensions. On the other hand, the integers~$\eta$ and~$\vartheta$ still need to be computed. Indeed, it would be interesting to study the complexity of the computation of these variables in detail, which we suspect to be a hard problem.
 \end{remark*}

\section{Generation of Cantorian tableaux}\label{sec:Generation}
First, recall that~$\mathcal{T}^s_n$ is totally ordered by the Parikh composition order on~$A^\star$ as each tableau can be considered as a vector of words of length~$n$ in~$A^\star$. 

\begin{definition}
 A tableau~$T$ is \emph{reduced} (or in its \emph{reduced form}) if its Parikh compositions~$\mathfrak{P}_T$ is equal to its Parikh compositions representative.
\end{definition}

Observe that there are many reduced tableaux in a given class. The next definition gives a canonical representative having convenient properties for computations.

\begin{definition}
 Let~$T\in\mathcal{T}^s_n$. If~$T\simc T'$ implies that~$T\blacktriangleleft T'$ for all~$T'\in\mathcal{T}^s_n$, then we call~$T$ a \emph{minimal reduced tableau} and we denote it by~$T^\blacktriangleleft$.
\end{definition}
By  definition of~$\blacktriangleleft$, such a tableau is indeed reduced and unique, which justifies its name. Since reduced tableaux are easy to obtain, their use significantly improves the computation of Cantorian tableaux. Nevertheless, passing from a reduced tableau to the minimal reduced tableau can be very costly.
\begin{remark*}
 Since computing the minimal reduced tableau from a reduced tableau can be cumbersome, it would be interesting to obtain an algorithm that yields the reduced form of a tableau which minimizes the number of reduced tableaux it may possibly produce.
\end{remark*}
In Section \ref{sec:EC}, we established a formula for the number of Cantorian tableaux in a given class. It is also possible to generate all tableaux in a class using the action of~$\mathbb{B}\rtimes\mathbb{S}$ on ~$\mathcal{T}^s_n$ defined by equation \eqref{eq:action}. All that remains to do  is to find all Cantorian class representatives of~$\mathcal{T}^s_n/{\simc}$. To do so, we consider all Parikh compositions representatives that respect the necessary conditions given in \cite{BMRR}:
\begin{condition}[Corollary 2 of \cite{BMRR}]\label{sparse}
  Let $T$ be an $n\times n$ tableau and suppose some letter, 
  say $a$, occurs at least  $n^2-n+1$ times  in $T$.  Then $T$ is
  non-Cantorian. More specifically
  \begin{equation*}
    a^n \in L \cap \Perm(T). 
  \end{equation*}
  If $a$ occurs only $n^2-n$ times, the result need not be true.
\end{condition}
To reduce the computations, we also use the following result in the case of a $2$-letter alphabet. 
\begin{theorem} [Theorem 7 of \cite{BMRR}]
The number $c(n,p)$ of Cantorian tableaux over  $A=\{a,b\}$
  with exactly $p$ occurrences of the letter $b$ is
\[   c(n,p) = 
   \begin{cases}
     0 &\text{ for } p<n,\\
     n &\text{ for } p=n\geq 3,\\
     0 &\text{ for } p=n+1\text{ and } n\geq 4,\\
     0 &\text{ for } p=n+2\text{ and } n\geq 5.
   \end{cases}
\]
\end{theorem}
For each Parikh compositions representative, we can build the corresponding distinct minimal reduced tableaux by using a recursive algorithm on the columns. Then, we test each such minimal reduced tableau and record the Cantorian ones. Finally, we compute the cardinality of each class using Equation~\eqref{eq:card_class}.\\

Below we list   the Cantorian minimal reduced representatives found for small dimensions and their class cardinality.\\

\noindent
Dimension~$n=2$ with~$s\geq2$:
{\footnotesize
\[
R^s_2  =  \begin{bmatrix}
                          1 & 1 \\
                          2 & 2 \\
                         \end{bmatrix} \quad,\quad 
|[R^s_2]|  =  s^2(s-1)^2 .
\]
}\\
\smallskip
Dimension~$n=3$,~$s=2$:
{\footnotesize
\[R^2_3  =  \begin{bmatrix}
                          1 & 1 & 1\\
                          1 & 1 & 1\\
                          2 & 2 & 2\\
                         \end{bmatrix} \quad,\quad 
|[R^2_3]|  =  24 .
\]
}\\
Dimension~$n=3$,~$s=3$:
{\footnotesize
\[
   \begin{array}{c}
          \begin{bmatrix}
                          1 & 1 & 1\\
                          1 & 1 & 1\\
                          2 & 2 & 2\\
                         \end{bmatrix} \\ \noalign{\medskip}
|[R_1]| = 648
   \end{array} ,
   \begin{array}{c}
          \begin{bmatrix}
                          1 & 1 & 1\\
                          1 & 1 & 2\\
                          2 & 2 & 3\\
                         \end{bmatrix} \\ \noalign{\medskip}
|[R_2]| = 1944	
   \end{array} ,
   \begin{array}{c}
          \begin{bmatrix}
                          1 & 1 & 1\\
                          1 & 2 & 2\\
                          2 & 3 & 3\\
                         \end{bmatrix} \\ \noalign{\medskip}
|[R_3]| = 1944
   \end{array} ,
   \begin{array}{c}
         \begin{bmatrix}
                          1 & 1 & 1\\
                          1 & 2 & 2\\
                          1 & 3 & 3\\
                         \end{bmatrix} \\ \noalign{\medskip}
|[R_4]| = 324
   \end{array} , 
   \begin{array}{c}
        \begin{bmatrix}
                          1 & 1 & 1\\
                          2 & 2 & 2\\
                          3 & 3 & 3\\
                         \end{bmatrix}\\ \noalign{\medskip}
|[R_5]| = 216  
   \end{array} .
\]}

\medskip
The following lemma is useful for  establishing a formula for Cantorian tableaux of fixed small dimensions and variable alphabet size.

\begin{lem}\label{lem:class}
If~$s>n$, then the Cantorian minimal reduced forms in~$\mathcal{T}^s_n$ are the Cantorian minimal reduced forms in~$\mathcal{T}^n_n$.
\end{lem}

\begin{proof}
Let~$c_i$ be the~$i$th column of a tableau~$T$. The maximal number of distinct letters that  may appear in~$c_i$ is~$n$. Using a bijection on~$c_i$, we can always obtain a new column-word~$c'_i$ using only the first~$n$ letters of $A$. Applying  this process to all column-words of~$T$, we obtain a new tableau~$T'$ equivalent to the first one using at most~$n$ distinct letters. Thus, every class representative has at most~$n$ letters.
\end{proof}

Table 1 lists the number of Cantorian minimal representatives and the number of tested tableaux obtained experimentally, thanks to the implementation of our method in the computer algebra system Sage \cite{sage}.
{\footnotesize
\begin{table}[h!t!b]
\begin{tabular}{c||c|c|c|c|c|c}\label{tab_nb_rep}
$n \backslash s$ & 2 & 3 & 4 & 5 & 6 &~$\cdots$ \\ \hline\hline 2 & 1/1 & 1/1 & 1/1 & 1/1 & 1/1 &~$\cdots$ \\ \hline 3 & 1/3 & 5/9 & 5/9 & 5/9 & 5/9 &~$\cdots$ \\ \hline 4 & 6/21 & 56/171 & 107/275 & 107/275 & 107/275 &~$\cdots$ \\ \hline 5 & 11/165 & 1873/12574 &  &  &  & 
\end{tabular}
%
%
\vspace{4pt}
\caption{Number of Cantorian classes of size~$n\times n$ on $s$ letters.}
\end{table}}

Using the cardinality formula given by Equation \eqref{eq:card_class} in Theorem \ref{thm:card_class}, we are able to extend the enumerative results listed in  Table 2  in \cite{BMRR}.

\begin{table}[h!t!b]\label{tab:can}
\tiny
\begin{tabular}{c||r|r|r|r|r|}
$n \backslash s$ & $2$ & $3$ & $4$ & $5$ &$\cdots$\\ \hline\hline 2 &~$1\cdot 2^2$ &~$2^2\cdot 3^2$ &~$3^2\cdot 4^2$ &~$4^2\cdot 5^2$&$\cdots$ \\ \hline 3 &~$3\cdot 2^3$ &~$47\cdot 2^2\cdot 3^3$ &~$207 \cdot 3^2\cdot 4^3$ &~$579\cdot 4^2\cdot 5^3$&$\bullet\bullet\bullet$  \\ \hline 4 &~$109\cdot 2^4$ &~$25036 \cdot2^2\cdot 3^4$ &~$\mathbf{803613}\cdot 3^2\cdot 4^4$ &~$\mathbf{9419224}\cdot 4^2\cdot 5^4$&$\bullet\bullet\bullet$\\ \hline 5 &~$2765\cdot 2^5$ &~$\mathbf{16304200}\cdot2^2 \cdot 3^{5}$ &  & &\\ \hline 6 &~$324781\cdot 2^6$ &  &  & &  \\ \hline 7 &~$37304106\cdot 2^7$ &  &  & &\\ \hline 8 &~$13896810621\cdot 2^8$ &  &  & & \\ \hline 9 &~$5438767247337\cdot 2^9$ &  &  & & \\ \hline 10 &~$6889643951630251\cdot 2^{10}$ &  &  &  &\\ \hline 11 &~$8135113082369752094\cdot 2^{11}$ &  &  &  &\\
\end{tabular}
\vspace{4pt}
\caption{Number of Cantorian tableaux of size $n\times n$ on $s$ letters.}
\end{table}

Moreover, once the Cantorian minimal reduced forms of~$\mathcal{T}^n_n$ are computed,  we deduce a closed formula for the number of Cantorian tableaux of dimension~$n$ on an alphabet of~$s$ letters. The next proposition  extends Theorem 1 of \cite{MMF}.

\begin{prop}\label{prop:Cantor234}
The number $C(n,s)$ of Cantorian tableaux for~$n=2,3$ and~$4$ is given by the following polynomials
\begin{align}
C(2,s) &= s^{2} \cdot (s - 1)^{2};\label{C2s} \\ 
C(3,s) &= s^{3} \cdot (s - 1)^{2} \cdot (s^{4} + 2 s^{3} - 15 s^{2} + 16 s - 1); \label{C3s}\\ 
C(4,s) &= s^{4} \cdot (s - 1)^{2} \cdot (s^{10} + 2 s^{9} + 3 s^{8} - 92 s^{7} - 43 s^{6} + 1014 s^{5} \label{C4s} \\  & \quad- 449 s^{4} - 5680 s^{3} + 12045 s^{2} - 9406 s + 2629). \nonumber
\end{align}
\end{prop}
\begin{proof}
By Lemma~\ref{lem:class} and Theorem \ref{thm:card_class}, it suffices to compute the representatives for $s\leq n$. 
Equation \eqref{C2s} was already established  in \cite{MMF} (Theorem 1,  p. 332). Nevertheless, we provide an alternate and much simpler proof.
For $n=s=2$, there is only one Cantorian class represented by the tableau $R^s_2$.
The cardinality of the Cantorian class is obtained using equation \eqref{eq:card_class}. It remains to compute that $\eta=0$ and $\vartheta=2$ which then yields 
\begin{equation*}
\frac{(2)^2\prod_{i=1}^{2}\frac{s!}{(s-2)!}}{(2+ 0)2}=\prod_{i=1}^{2}\frac{s!}{(s-2)!}=s^2(s-1)^2.
\end{equation*}
For $n=3$ and $s\leq 3$, there are 5 different representatives denoted previously by $R_1,R_2,R_3,R_4$ and~$R_5$. Again, the cardinalities of Cantorian classes are obtained from equation \eqref{eq:card_class} by computing $\eta$ and $\vartheta$ for each representative:

\begin{align*}
|[R_1]| & =  \frac{(3!)^2\frac{s!s!s!}{(s-2)!(s-2)!(s-2)!}}{(2!\cdot 3!+ 0)1}= 3s^3(s-1)^3; \\
|[R_2]| & =  \frac{(3!)^2\frac{s!s!s!}{(s-2)!(s-2)!(s-3)!}}{(1\cdot 2!+ 0)2}= 9s^3(s-1)^3(s-2); \\
|[R_3]| & =  \frac{(3!)^2\frac{s!s!s!}{(s-2)!(s-3)!(s-3)!}}{(1\cdot 2!+ 0)2}= 9s^3(s-1)^3(s-2)^2; \\
|[R_4]| & =  \frac{(3!)^2\frac{s!s!s!}{(s-1)!(s-3)!(s-3)!}}{(1\cdot 2!+ 0)6}= 3s^3(s-1)^2(s-2)^2; \\
|[R_5]| & =  \frac{(3!)^2\frac{s!s!s!}{(s-3)!(s-3)!(s-3)!}}{(1\cdot 3!+ 0)6}= \left(s(s-1)(s-2)\right)^3.
\end{align*}
The sum of these polynomials yields equation~\eqref{C3s}. For equation~\eqref{C4s}, the reader may proceed similarly with the 107 representatives  or use, for instance, the computer algebra system Sage \cite{sage}.
\end{proof}

\begin{remark*}
It would be interesting to exhibit  a general recursive construction: indeed, it amounts to determine the distinct minimal representatives of size $n\times n$ from its minors of size $(n-1)\times(n-1)$. Furthermore, in \cite{MMF}, it is shown that $C(n,s)\leq (s^n-s)^n$ which implies that if $C(n,s)=s^n(s-1)^2(s^{n^2-n-2}+a\cdot s^{n^2-n-3}+\cdots)$, then necessarily $a\leq 2$.
\end{remark*}

\section{Bi-Cantorian tableaux}\label{sec:BiC}

In Section 7 of \cite{BMRR}, the authors introduced the subclass of bi-Cantorian tableaux: 

\begin{definition}
 A tableau~$T\in\mathcal{T}^s_n$ is bi-Cantorian if~$\Perm(T)\cap(C\cup L)=\emptyset$.
\end{definition}
In other words, a tableau is bi-Cantorian if it is Cantorian and, moreover, if none of the column-words appear in its permanent. Taking into account that we can generate Cantorian tableaux using the representatives, we can now compute by brute force the bi-Cantorian tableaux. Table 3 lists the first results.
\begin{table}[h!t!b]\label{tab:bican}
\begin{small}
\begin{tabular}{c||c|c|c|c|c|c}
$n \backslash s$ & 2 & 3 & 4 & 5 & 6 &~$\cdots$ \\ \hline\hline 2 &~$1\cdot 2\cdot 1$ &~$2\cdot 3\cdot 3$ &~$3\cdot 4\cdot 7$ &~$4\cdot 5\cdot 13$ &~$5\cdot 6\cdot 21$ &~$\cdots$ \\ \hline 3 &~$1\cdot 2\cdot 3$ &~$2\cdot 3 \cdot 367$ &~$3\cdot 4\cdot 6179$ &~$4\cdot 5\cdot 43065$ &  &  \\ \hline 4 &~$1\cdot 2\cdot 91$ &~$2\cdot 3 \cdot 402873$ &  &  &  &  \\ \hline 5 &~$1\cdot 2\cdot 2005$ &  &  &  &  & \\
\end{tabular}
\end{small}
\vspace{4pt}
\caption{Number of bi-Cantorian tableaux of size $n\times n$ on $s$ letters.}
\end{table}
The property of being bi-Cantorian is not invariant under the relation~$\simc$. Besides this fact, observe
that given a permutation~$\sigma$ of~$S_n$, the property of being bi-Cantorian is invariant under the action of~$\sigma$ on both the rows and columns. It is  also invariant under the action of pairs $(\sigma,\tau)\in \mathbb{S}$ such that $ \sigma\cdot C\cap \Perm(T\cdot \tau)=\emptyset$, and of course,  by any bijection on $A$ as well. This defines a new equivalence relation~$\simb$ on the set~$\mathcal{T}^ s_n$.

Given the relation~$\simb$, for dimension~$n=2$, there are three bi-Cantorian classes:
\[
\begin{bmatrix}
                          1 & 2 \\
                          2 & 1 \\
                         \end{bmatrix},
\begin{bmatrix}
                          1 & 2 \\
                          2 & 3 \\
                         \end{bmatrix},
\begin{bmatrix}
                          1 & 2 \\
                          3 & 4 \\
                        \end{bmatrix};
\]
while for~$n=3$ and~$s=2$, there is only one class represented by the tableau 
\[\begin{bmatrix}
                          1 & 1 & 2 \\
                          1 & 1 & 2 \\
                          2 & 2 & 1 \\
                         \end{bmatrix} .
\]
For~$n=3$ and~$s=3$, there are 32 classes and for $s=4$, there are 173 classes. 

The choice for a bi-Cantorian class representative is not as clear as for Cantorian classes because of the rather complicated relation~$\simb$. In trying to understand bi-Cantorian tableaux, we relate this classification with some coloring of a graph.

Let~$K(s)$ be the set of~$s$-colored cycle graphs on 4 vertices (labeled~$v_1,v_2,v_3,v_4$ clockwise) such that no edge has the same color on both of its vertices. Moreover, let~$B(s)$ be the set of 2$\times$2 bi-Cantorian tableaux on~$s$ letters.

\begin{prop}
  The function  $ \psi :  B(s)  \rightarrow  K(s) $ defined by
\begin{equation*}   \begin{bmatrix}
                          \alpha_{1,1} & \alpha_{1,2} \\
                          \alpha_{2,1} & \alpha_{2,2} \\
                         \end{bmatrix} \mapsto  \{(v_1,\alpha_{1,1}), (v_2,\alpha_{1,2}), (v_3,\alpha_{2,2}), (v_4,\alpha_{2,1})\} \\
  \end{equation*}
  is a bijection between~$B(s)$ and~$K(s)$. In particular, we have
  \[|K(s)|=|B(s)|=2\binom{s}{2}+12\binom{s}{3}+24\binom{s}{4}=s(s-1)(s^2-3s+3),\] and the $3$ bi-Cantorian representatives give the non-isomorphic colorings of the cycle graph.
\end{prop}

\begin{proof}
First, we prove that the image of a bi-Cantorian tableau gives a proper coloring of the 4-cycle. So, consider a bi-Cantorian tableau~$B$ and its image $\{(v_1,\alpha_{1,1}), (v_2,\alpha_{1,2}), (v_3,\alpha_{2,2}), (v_4,\alpha_{2,1})\}$. If the associated colored 4-cycle has an edge which is monochromatic with color~$\alpha\in A$, this means that the bi-Cantorian tableau~$B$ has the word~$\alpha^2$ in the set~$C\cup L$. Now, suppose that~$\alpha^2$ is a row-word of~$B$. The diagonal of~$B$ actually appears as the first or second column-word of~$B$, meaning that~$B$ is not bi-Cantorian, which is a contradiction. If~$\alpha^2$ is a column-word, the same argument applies verbatim by interchanging rows and columns.
The injectivity of this function is easily verified. It remains  to check that~$\psi$ is surjective. Given a proper coloring of the 4-cycle~$K=\{(v_1,\beta_1), (v_2,\beta_{2}), (v_3,\beta_3), (v_4,\beta_4)\}$, with~$\beta_{i=1,\dots,4}\in A$, the inverse image gives the tableau
\[ \psi^{-1}(K)=\begin{bmatrix}
                          \beta_1 & \beta_2 \\
                          \beta_4 & \beta_3 
                         \end{bmatrix}.
\]
  To  check that this tableau is indeed bi-Cantorian, we proceed by contradiction and assume that the tableau is not bi-Cantorian. Then, there is a column-word or row-word~$w$ which is either equal to the diagonal of~$\psi^{-1}(K)$ or its secondary diagonal. Suppose that~$w$ is equal to the main diagonal-word~$\beta_1\beta_3$, since~$\beta_2$ and~$\beta_4$ are not equal to~$\beta_1$ nor to~$\beta_3$,~$w$ cannot be equal to the second column or second row. The same holds for the first row-word and first column-word. Finally, one can see that if~$w$ is equal to the secondary diagonal, the same arguments hold and lead to a contradiction, forcing the tableau to be bi-Cantorian.
                         
To get the cardinality of~$K(s)$, one can easily compute the number of distinct colorings by  using 2, 3 and 4 distinct colors, which are 2, 12 and 24 respectively. Finally, starting with~$s$ colors, it remains to count the number of ways to choose 2, 3 and 4 distinct colors within~$s$ colors.
\end{proof}

Looking at the enumeration of bi-Cantorian tableaux, some new questions arise. Denote by~$B(n,s)$ the number of~$n\times n$ bi-Cantorian tableaux on~$s$ letters and by~$C(n,s)$ the number of~$n\times n$ Cantorian tableaux on~$s$ letters. First, in order to obtain asymptotic results similar to Section $5$ of \cite{BMRR}, one should find a sufficient condition for  not being bi-Cantorian. 




Naturally, we can also ask what is the ratio of bi-Cantorian tableaux to Cantorian tableaux, and whether the limit as $n\rightarrow \infty$ exists, if so what is its value?
Using the data from Table 2 and Table 3 with ~$s=2$,  the first values of this ratio are
$$\frac{B(n,2)}{C(n,2)}= 0.5, 0.25, 0.104, 0.045$$ for respectively~$n=2,3,4,5$. This strongly suggests that the ratio tends to zero.  In the next section, we attempt  to answer this question by using an alternate approach.

\section{Variations and extensions of Cantorian tableaux}\label{sec:Variations}

Consider an infinite tableau formed by listing the algebraic numbers in some base~$s$ in rows. Are there algebraic columns? If yes, how many? In this section, we give results that would suggest the right answer is no with high probability and discuss further extensions of Cantorian tableaux.

Consider the infinite tableau~$T^{\infty}$ formed by row-words in the set~$L^{\infty}=A^\star \alpha_1^{\omega} \cup A^\star \alpha_2^{\omega} \cup \cdots \cup A^\star \alpha_s^{\omega}$. Every row in~$T^ {\infty}$ finishes with a tail of~$\alpha_1$ or a tail of~$\alpha_2$, etc. In other words, for every row-word~$\ell$, we have
\begin{align*}
|\ell|_{\alpha_1}&=\infty,  |\ell|_{\alpha_2}<\infty,  |\ell|_{\alpha_3}<\infty, \text{ etc.} \\
\text{or } |\ell|_{\alpha_2}&=\infty,  |\ell|_{\alpha_1}<\infty,  |\ell|_{\alpha_3}<\infty, \text{ etc.} \\
\text{or } |\ell|_{\alpha_3}&=\infty,  |\ell|_{\alpha_1}<\infty,  |\ell|_{\alpha_2}<\infty, \text{ etc.} \\
 & \dots & \\
\text{or } |\ell|_{\alpha_s}&=\infty,  |\ell|_{\alpha_1}<\infty,  |\ell|_{\alpha_2}<\infty, \text{ etc.} 
\end{align*}

\begin{theorem}\label{thm:biCant}
The tableau~$T^{\infty}$ is Cantorian. Moreover, if~$s=2$, then it is not bi-Cantorian.
\end{theorem}

\begin{proof}
First, we prove that every element~$p\in \Perm(T^{\infty})$ contains infinitely many times each letter~$\alpha\in A$. To prove this, we adapt the proof of Theorem 10 in~\cite{BMRR}. By contradiction, assume that $|p|_a<\infty$  for some $a\in A$. Then, let $\mu$ be the morphism without fixed point defined by $\mu(b)=a$ for all $b\not =a$, and $\mu(a)=c$ for some $c\not = a$. By Theorem 3 in \cite{BMRR}, we have $\Perm(\mu T^{\infty})\cap L^{\infty}=\emptyset$, which implies that ~$\mu(p)$ contains infinitely many occurrences of letters $a$ and $c$. But from some index on, $p$ does not contain the letter $a$, so that $\mu(p)$ has an infinite tail of $a$'s. Contradiction. We can therefore say that~$L^{\infty}\cap\Perm(T^{\infty})=\emptyset$ and hence~$T^{\infty}$ is Cantorian.

One can see that there are infinitely many rows beginning with~$\alpha_1$, infinitely many rows beginning with~$\alpha_2$, etc. Thus, the first column contains infinitely many times each letter in~$A$. This property is true for every column of~$T^{\infty}$. If we suppose that~$s=2$, then Theorem 11 in \cite{BMRR} says that~$\Perm(T^{\infty})=[0,1]\setminus L^{\infty}=\{\text{words with infinitely many } \alpha_1\text{ and }\alpha_2\}$. This latter set is exactly the set~$C^{\omega}$ of column-words of~$T^{\infty}$. Therefore,~$\Perm(T^{\infty})\cap C^{\omega}=C^{\omega}\neq\emptyset$ and hence~$T^{\infty}$ is not bi-Cantorian.
\end{proof}

In the statement of Theorem 10 in \cite{BMRR}, it is assumed that the set $L$ of row-words  contains the rational numbers~$\rat$. The above tableau~$T^{\infty}$, with~$s=2$, is the same as the one containing the binary expansions of the numbers~$k/2^n$, with~$k,n\in \nat$ and~$0\leq k/2^n<1$, with the convention that every number should appear twice, once with a tail of 0's and once with a tail of 1's.
Our last proof implies that a more general statement of this theorem is true, namely:

\begin{corollary} Let  the set $L$ of row-words  of $T^{\infty}$ be such that $L^{\infty}\cap [0,1] \subseteq  L\subseteq [0,1]$.   Then,  $T^{\infty}$ is Cantorian.
\end{corollary}

\begin{remark*}
 In view of Theorem 13 in \cite{BMRR}, Theorem \ref{thm:biCant} above suggests that there is a very small probability for a column-word to be algebraic.
\end{remark*}

When discussing infinite tableaux in  previous articles, Brlek  \emph{et al.}  \cite{BMRR} and Mend\`es  France \cite{MMF} were mostly concerned with real numbers and their expansion in basis~$s\geq 2$. Other kind of expansions may be interesting to look at, in particular continued fractions. We discuss for the remaining of this section the very special case of irrational formal power series $\sum_{n\geq 1} \frac{\alpha_n}{x^n}$ over the field $\mathbb{F}_2$ with two elements 0 and 1. It is well known that these series have a continued fraction representation $[0,A_1(x),A_2(x), \dots]$ where for all~$j\geq 1$,~$A_j(x)$ are polynomials of degrees $\geq 1$. In a remarkable paper \cite{BS}, Baum and Sweet study those continued fractions where for all~$j$, the degree of~$A_j(x)$ is 1, i.e.~$[0, x+a_1, x+a_2, \dots ]$, with~$a_j=0 \text{ or }1$, which for short we call $BS$-elements. Baum and Sweet observe that there exist countably many algebraic $BS$-elements. We show:

\begin{prop}
Let~$T$ be the infinite tableau where the rows represent the partial quotients~$x+a^j_i$ of the family of algebraic~$BS$-elements. Then,~$T$ is Cantorian and more precisely $~\Perm (T)$ represents the family of all transcendental~$BS$-elements.
\end{prop}

\begin{proof}
Identify~$x+a^j_i$ with~$a^j_i$. The proposition is then a trivial corollary of Theorem 11 of \cite{BMRR}.
\end{proof}

\section{Cantorian-type tableaux and colored hypergraphs}\label{sec:HyperGraphs}

In this section, we link the study of Cantorian tableaux to colored hypergraphs with a square number of vertices. A hypergraph~$H$ is a pair~$(V,B)$ consisting of a \emph{vertex} set~$V$ and a family~$B$ of subsets of $V$, called \emph{blocks}. A hypergraph is \emph{regular} if each vertex of~$V$ appears in the same number of blocks. Also, a hypergraph is \emph{uniform} if every block in~$B$ contains the same number of vertices. Two hypergraphs are \emph{isomorphic} if there is a bijection between their vertex set  preserving the blocks.

We now build a  hypergraph on~$n^2$ vertices. Set~$V=\{v_{ij}| 1\leq i,j \leq n\}$ for the vertex set. The block family consists of two distinct sets of blocks: the first one is $L=\left\{ \{v_{ij}|1\leq j\leq n\}| 1\leq i\leq n \right\}$, called \emph{row blocks}; the second block-set is~$P=\left\{ \{v_{\pi(i)i}| 1\leq i\leq n\}  |\pi\in S_n\right\}$, called  \emph{diagonal blocks}. We set~$B=L\cup P$. Thus,~$H=(V,B)$ is a $((n-1)!+1)$-regular and~$n$-uniform hypergraph. Each block of this hypergraph has a natural linear ordering of its vertices according to the second index. A vertex coloring~$\chi$ of~$H$ is a map from~$V$ to a color set~$A$, with~$|A|=s$. Such a coloring~$\chi$ of~$H$ is \emph{intersecting} if there exists a sequence of colors~$(\alpha_1,\dots,\alpha_n)$, with~$\alpha_i\in A$, that appears both in a block of $L$ and in a block of~$P$. 
A non-intersecting colored hypergraph is called for obvious reasons  \emph{Cantorian} since it translates literally to a tableau. Indeed,  if $\chi$ and $\chi'$ are two vertex colorings of $H$,  they are  \emph{isomorphic} if  there exists a bijection $\lambda:\chi(L)\cup\chi(P) \longrightarrow\chi'(L)\cup\chi'(P)$ such that its restriction $\lambda:\chi(L)\cap\chi(P) \longrightarrow\chi'(L)\cap\chi'(P)$ is also a bijection. Clearly, such a $\lambda$ leaves the Cantorian property invariant.

\begin{prop}
If two tableaux are $\simc$-equivalent, then their corresponding colored hypergraphs are isomorphic.
\end{prop}

\begin{proof}
Let $T$ and $T'$ be two equivalent tableaux and let $H_T$ and $H_{T'}$ be their respective colored hypergraphs. Clearly, acting by permutation of rows and columns on~$T$ yields an isomorphic colored hypergraph. It remains to verify the case where~$T'$ is obtained from $T$ by a permutation of the alphabet on a certain column. In this case, the permutation induces the bijection $\lambda:\chi(L)\cup\chi(P) \longrightarrow\chi'(L)\cup\chi'(P)$ so that $\lambda$ is also a bijection on their intersection. Therefore, the new hypergraph $H_{T'}$ is isomorphic to $H_T$.
\end{proof}

\begin{remark*}
The converse of this statement is false. Consider the three non-equivalent tableaux on the alphabet $\{1,2,3\}$ which yield isomorphic hypergraphs:
{\footnotesize
\[
   \begin{array}{c}
          \begin{bmatrix}
                          1 & 1 & 1\\
                          1 & 2 & 2\\
                          1 & 3 & 3\\
                         \end{bmatrix}
   \end{array} ,
   \begin{array}{c}
         \begin{bmatrix}
                          1 & 1 & 1\\
                          1 & 2 & 2\\
                          2 & 3 & 3\\
                         \end{bmatrix}
   \end{array} , 
   \begin{array}{c}
        \begin{bmatrix}
                          1 & 1 & 1\\
                          2 & 2 & 2\\
                          3 & 3 & 3\\
                         \end{bmatrix}
   \end{array} .
\]}

\noindent Every tableau has $3$ distinct row blocks and $6$ distinct diagonal blocks. The bijection $\lambda$ sends corresponding blocks to each other. Since all tableaux are Cantorian,~$\lambda$ is trivially a bijection on $\chi(L)\cap\chi(P)=\varnothing$. Thus, the isomorphism classes of Cantorian hypergraphs define an equivalence  relation $\simh$ which is coarser than  $\simc$. We conjecture that $\simh$  is the coarsest equivalence relation on Cantorian tableaux which  could improve the computations. Furthermore, it is possible to extend the notion of Cantorian hypergraph by adding another family~$C$ \emph{orthogonal} to~$L$ to represent the columns of a tableau. This leads to similar notions for bi-Cantorian tableaux classes and isomorphism classes of bi-Cantorian hypergraphs. It would be interesting to study the isomorphism classes of bi-Cantorian hypergraphs in order to give a simpler description of bi-Cantorian classes.
\end{remark*}

%
\paragraph*{\emph{Acknowledgement}} \SB~ is supported by a research grant from NSERC Canada which partially contributed to \JPL's stay in Gradignan, where this paper was brought to its present form in the delightful estate of \MMF. We are grateful to the anonymous reviewers for their careful reading and for the comments provided.

\end{document}